\documentclass[11pt]{amsart}
\usepackage{amsmath}
\usepackage{amsfonts}
\usepackage{amsthm}
\usepackage{amssymb}
\usepackage{amscd}
\usepackage[all]{xy}
\usepackage{color}
\usepackage{ulem}
\newcommand{\so}{\boldsymbol s}
\newcommand{\ra}{\boldsymbol r}

\newcommand{\mpnr}{M_{\mathcal P,\vec n}(R)}
\newcommand{\mpnz}{M_{\mathcal P,\vec n}(\mathbb Z)}
\newcommand{\mpmr}{M_{\mathcal P,\vec m}(R)}

\newcommand{\glpnr}{\text{GL}_{\mathcal P,\vec n}(R)}

\newcommand{\slpnr}{\text{SL}_{\mathcal P,\vec n}(R)}

\newcommand{\glpnc}{\text{GL}_{\mathcal P,\vec n}(\C)}
\newcommand{\glpmc}{\text{GL}_{\mathcal P,\vec m}(\C)}

\newcommand{\glpnz}{\text{GL}_{\mathcal P,\vec n}(\Z )}
\newcommand{\glpmz}{\text{GL}_{\mathcal P,\vec m}(\Z )}
\newcommand{\glpmq}{\text{GL}_{\mathcal P,\vec m}(\Q )}
\newcommand{\glpnq}{\text{GL}_{\mathcal P,\vec n}(\Q )}

\newcommand{\slpnz}{\text{SL}_{\mathcal P,\vec n}(\Z )}
\newcommand{\slprz}{\text{SL}_{\mathcal P,\vec r}(\Z )}
\newcommand{\slpmz}{\text{SL}_{\mathcal P,\vec m}(\Z )}

\newcommand{\mpmnr}{M_{\mathcal P,\vec m, \vec n}(R)}

\newcommand{\mpnnr}{M_{\mathcal P,\vec n, \vec n}(R)}

\newcommand{\mpmnc}{M_{\mathcal P,\vec m, \vec n}(\C)}
\newcommand{\mpmnz}{M_{\mathcal P,\vec m, \vec n}(\Z )}
\newcommand{\mpmnq}{M_{\mathcal P,\vec m, \vec n}(\Q )}

\newcommand{\glnc}{GL_n(\mathbb C)}

\newcommand{\glmc}{GL_m(\mathbb C)}

\newcommand{\Z}{\mathbb Z}
\newcommand{\Q}{\mathbb Q}

\newcommand{\C}{\mathbb C}

\numberwithin{equation}{section}
\newtheorem{theorem}[equation]{Theorem}
\newtheorem{corollary}[equation]{Corollary}

\newtheorem{lemma}[equation]{Lemma}

\theoremstyle{definition}
\newtheorem{definition}[equation]{Definition}

\newtheorem{remark}[equation]{Remark}
\newtheorem{example}[equation]{Example}

\title[Decidability of flow equivalence and isomorphism problems]{Decidability of flow equivalence and isomorphism problems for graph $C^*$-algebras and quiver representations}

\author{Mike Boyle\and Benjamin Steinberg}

\address{
Department of Mathematics, University of Maryland, College
Park, MD 20742-4015, USA
}
\email{mmb@math.umd.edu}
\thanks{This work was supported  by
  the Danish National Research Foundation through the Centre for Symmetry and Deformation (DNRF92). This work was partially supported by a grant from the Simons Foundation (\#245268 to Benjamin Steinberg). We also thank the anonymous
referee for comments which led to an improved paper.}

\address{
  Department of Mathematics, City College of New York,
  Convent Avenue at 138th Street, New York, New York 10031, USA
}
\email{bsteinberg@ccny.cuny.edu}


\date{\today}


\begin{document}

\maketitle

\begin{abstract} We note that the deep results of Grunewald and Segal on algorithmic problems for arithmetic groups
  imply the decidability of several matrix equivalence problems involving
  poset-blocked matrices over $\Z$.
  Consequently, results of Eilers, Restorff, Ruiz and
  S{\o}rensen imply that isomorphism and stable isomorphism of unital
  graph   $C^*$-algebras (including the Cuntz-Krieger algebras) are
  decidable. One can also decide flow equivalence for shifts of finite type,
  and isomorphism of $\Z$-quiver representations (i.e.,
  finite diagrams of homomorphisms of
  finitely generated abelian groups).
  \end{abstract}

\section{Introduction}
This paper concerns algorithmic decidability questions in symbolic dynamics and $C^*$-algebras.  Recall that, up to conjugacy, a shift of finite type (SFT) is given by an $n\times n$ $0/1$-matrix $A$.  The corresponding subshift $\mathcal X_A$ consists of all bi-infinite sequences $(x_i)_{i\in \mathbb Z}$ over the alphabet $\{1,\ldots, n\}$  such that $A_{x_i,x_{i+1}}=1$. We recall that two subshifts are flow equivalent if their suspensions (or mapping tori) are conjugate,
modulo a time change, as flows over $\mathbb R$.
One can then ask the natural algorithmic question: given $A$ and $B$ square $0/1$-matrices (of possibly different sizes), decide whether the corresponding subshifts $\mathcal X_A$ and $\mathcal X_B$ are flow equivalent.
Parry and Sullivan \cite{parrysullivan}, and then Bowen and
Franks \cite{BowenFranks}, provided fundamental matrix invariants for this
problem; Franks \cite{jf:fesft} gave complete invariants for the irreducible
case; and Huang (unpublished) found complete matrix
invariants for the general case. A thorough  treatment of complete invariants
was given in \cite{mb:fesftpf,BoyleHuang}, but the question of decidability was left open.
This paper provides the finishing touch on deciding this question.  Note that it is still an open question to decide whether two shifts of finite type are conjugate.  See~\cite{MarcusandLind} for background on symbolic dynamics.

Cuntz and Krieger~\cite{cuntzkrieger}, motivated in part by the study of flow equivalence of shifts of finite type, introduced a very
important class of $C^*$-algebras associated to square nondegenerate
$0/1$-matrices.
These were generalized to another important class,
the class of graph $C^*$-algebras, defined as
follows.\begin{footnote}{Following \cite{errs:complete}, we use
  the notation and definition of \cite{FLR00}, not
  \cite{Raeburn2005}.}\end{footnote}
Let $E=(E^0,E^1)$ be a directed graph,
with vertex set $E^0$ and edge set $E^1$, allowed to be finite or
countably infinite,
and
with source and range maps
$\so,\ra\colon E^1\to E^0$.
The graph $C^*$-algebra $C^*(E)$ is the universal $C^*$-algebra
generated by a set
$\{p_v: v\in E^0\}$
of mutually orthogonal projections
 and
a set $\{s_e: e\in E^1\}$
of partial isometries,
satisfying (for all $e,f$ in $E^1$ and $v$ in $E^0$) the relations
\begin{alignat*}{2}
s_e^*s_e & =p_{\ra(e)}\ ,  \quad \quad &
s_e^*s_f & =0 \quad \text{if } e\neq f \ ,\\
s_es_e^* & \leq p_{\so (e)} \ ,  \quad \quad &
p_v & =\sum_{e \in \so^{-1}(v)}s_es_e^*\quad \text{if }0 < |\so^{-1}(v)| < \infty \ .
\end{alignat*}
The graph $C^*$-algebras isomorphic to Cuntz-Krieger algebras are those
with $E^0$ and $E^1$ finite and with $\so^{-1}(v)\neq \emptyset$ for all $v$.
A graph $C^*$-algebra is unital (i.e., possesses a unit) if and only if
its vertex set $E^0$ (but not necessarily $E^1$) is finite.
When $A,B$ are adjacency matrices of finite
graphs defining Cuntz-Krieger algebras,
flow equivalence of the SFTs defined by $A$ and $B$ implies
stable isomorphism of these algebras.
(Recall that $C^*$-algebras are stably isomorphic
if they become isomorphic upon tensoring with the algebra of compact
operators; for  separable $C^*$-algebras
this is the same as (strong) Morita equivalence in the sense of Rieffel.)
This connection,
made in \cite{cuntzkrieger}, is the heart of a fruitful interaction
between symbolic dynamics and the study of
Cuntz-Krieger algebras. The interaction in the case of general graph
$C^*$-algebras is less direct but still signficant. For much more on
these algebras and their classification, we refer to
the discussion and references of \cite{errs:complete}.


Now, let $(\mathcal P,\preceq)$ be a finite poset.  Without loss of generality, we  assume that $\mathcal P=\{1,\ldots, N\}$ and that $i\preceq j$ implies $i\leq j$.  Let $\vec n=(n_1,\ldots, n_N)$ be an $N$-tuple of positive integers and put $|\vec n|=n_1+\cdots+n_N$.  For any  ring $R$ with unit, define
$M_{\mathcal P,\vec n}(R)$ to be the $R$-subalgebra of $M_{|\vec n|}(R)$ consisting of all $|\vec n|\times |\vec n|$-matrices over $R$ with a block form
\[
M = \begin{pmatrix} M_{1,1} & \cdots & M_{1,N}\\ \vdots & \ddots &\vdots\\ M_{N,1} & \cdots & M_{N,N}\end{pmatrix}\] with each $M_{i,j}$ an $n_i\times n_j$-matrix over $R$ and
such that $M_{i,j}\neq 0$ implies $i\preceq j$; in particular, $M$ is block upper triangular. For example, if each $n_i=1$, then $M_{\mathcal P,\vec n}(R)$ is the usual \emph{incidence algebra} of $\mathcal P$~\cite{Stanley}.

We denote by $\glpnr$  the group of units of
$\mpnr$.
If $R$ is commutative, then $\slpnr$ is
the subgroup of matrices $M$ for which each diagonal block
$M_{ii}$ has determinant $1$.  For a subgroup $\Gamma$ of
$\glpnr$, two matrices $A,B\in M_{\mathcal P,\vec n}(R)$ are said to be
$\Gamma$-equivalent  if there are matrices $U,V\in \Gamma$
with $UAV=B$.

A vast collection of related works
followed the introduction of the Cuntz-Krieger $C^*$-algebras
in \cite{cuntzkrieger}, in particular, including many papers
on  graph $C^*$-algebras (consider the citations of
\cite{Raeburn2005}).
Eventually, Restorff in \cite{Restorff06}
showed that decidability of stable isomorphism for Cuntz-Krieger algebras satisfying Condition II of Cuntz reduces to deciding whether two matrices $A,B\in M_{\mathcal P,\vec n}(\mathbb Z)$ are $GL_{\mathcal P,\vec n}(\Z)$-equivalent.
In related work,
Boyle and Huang reduced the question of deciding flow equivalence for
shifts of finite type to deciding whether two matrices
$A,B\in M_{\mathcal P,\vec n}(\Z)$ are $SL_{\mathcal P,\vec n}(\Z)$-equivalent. 
  (See \cite{mb:fesftpf}; Huang's alternate development was never published.
  We provide additional detail in the Appendix.)
Eilers, Restorff, Ruiz and S{\o}rensen 
in 
  \cite{errs:complete}
 reduced the decidability of stable isomorphism
of unital graph $C^*$-algebras
(a class including the Cuntz-Krieger $C^*$-algebras)  to a  more
general problem of poset blocked equivalence of rectangular
matrices, which we describe later. They also reduced decidability
of isomorphism of unital graph $C^*$-algebras to yet another
matrix equivalence problem.

We shall point out that all these matrix equivalence
problems are decidable, by the
deep results of Grunewald and Segal~\cite{GrunewaldSegal80}, making the work of Borel and Harish-Chandra~\cite{BoHar62} effective.
We also  use Grunewald-Segal \cite{GrunewaldSegal80}
to prove the decidability of
isomorphism of  explicitly given commuting diagrams of
homomorphisms of finitely many finitely generated abelian groups.
 (This can be
interpreted as  decidability of isomorphism of $\Z$-quiver
representations.)
This applies to diagrams arising as
{\it reduced $K$-webs}
 in
the study of flow equivalent SFTs
\cite{BoyleHuang} and to related diagrams of
{\it reduced filtered K-theory} arising in the study of
certain $C^*$-algebras
(e.g. \cite{arr,Restorff06}).

These decidability results,
as proved by appeal to
\cite[Algorithm A]{GrunewaldSegal80},
do not provide practical decision procedures;
the extremely general  Algorithm A
is not even proved to be primitive recursive.

\section{Grunewald and Segal}
Following Grunewald and
Segal~\cite{GrunewaldSegal80},
by a $\mathbb Q$-group
we mean
a subgroup $J$ of $GL_n(\mathbb C)$ (for some $n\geq 1$)
which is the set of common zeros in
$GL_n(\mathbb C)$ of finitely many polynomials, with rational coefficients,
in the $n^2$ matrix entries.
The $\Q$-group is {\it given explicitly} if these polynomials
are explicitly given\begin{footnote}{However, when  such a
  set of polynomials
 is explicitly given, we usually will not write one out.}\end{footnote}.
If $R$ is a subring of $\mathbb C$, then $J_R=GL_n(R)\cap J$.
When there exists  an explicitly given
linear isomorphism, defined over $\Q$, from a $\Q$-group $J$ of complex
matrices onto a $\Q$-group $J'$, which carries $J_{\Z}$ onto
$(J')_{\Z}$, we may avoid mention of the isomorphism and simply
refer to $J$ as a $\Q$-group. For example, if
 $J_1$ and $J_2$ are $\Q$-groups,
then their direct product $J_1\times J_2$ is a $\mathbb Q$-group,
via the embedding $(A,B)\mapsto
\left( \begin{smallmatrix} A&0\\0&B
\end{smallmatrix}\right)$.

A \emph{rational action} of a $\Q$-group $J$ is a homomorphism
$\rho$ from $J$ into the group of permutations
of a subset $W$ of some complex vector
space $\C^m$ such that for each $w\in W$, the
coordinate entries of
the vector $\rho_g(w)$ are rational functions of the $n^2$ entries
of the matrix $g$ as $g$ runs through the identity component $J^0$ of
$J$; and for $w\in W\cap \Q^m$, these rational functions
are ratios of polynomials with
 rational coefficients.
 The action is \emph{explicitly given} if
for each $w\in W\cap \Q^m$,
(i) there is an effective procedure
which produces those coordinate rational functions,  and
(ii) for each $g\in J_{\Z} $ the vector $\rho_g(w)$ is effectively
computable.

By an \emph{arithmetic subgroup} of $J$, we mean a subgroup $\Gamma\leq J_{\Z}$ of finite index (usually, one allows a subgroup commensurable with $J_{\Z}$, but as pointed out in~\cite{GrunewaldSegal80} it is enough to consider finite index subgroups by performing a rational change of basis).  If $J$ is an explicitly given $\mathbb Q$-group, following Grunewald and Segal, we say that the arithmetic subgroup $\Gamma$ is \emph{explicitly given} if an upper bound on the index of $\Gamma$ in $J_\Z$ is given and an effective procedure is given to decide, for each $g\in J_\Z$, whether or not $g\in \Gamma$.  Most of this paper will only consider $\Gamma=J_\Z$, with the exception of Lemma~\ref{cokernellemma}.


The following stunning result is due to
Grunewald and
Segal~\cite[Algorithm~A]{GrunewaldSegal80}.


\begin{theorem} [Grunewald/Segal]\label{grunewaldsegal}
  Let $J$ be an explicitly given $\Q$-group and $\rho$ an explicitly
  given rational action of $J$ on a subset $W$ of $\C^m$. Let $\Gamma$ be an explicitly given arithmetic subgroup of $J$ (e.g., $\Gamma=J_\Z$).
  There is an algorithm, which
given vectors $v,w\in W\cap \mathbb Q^m$, decides
whether there exists $g\in \Gamma$ such that $\rho_g(v)=w$ (and
produces such a $g$, when one exists).
\end{theorem}

\begin{remark}
It is important to note that, implicit in~\cite[Algorithm A]{GrunewaldSegal80},  is that $J$, $\rho$ and $\Gamma$ should be considered part of the input (this is the point of $J$ and $\Gamma$ being ``explicitly given''), and not just the vectors $v,w\in W\cap \mathbb Q^m$, despite the wording of Theorem~\ref{grunewaldsegal} (which mimics that of~\cite[Algorithm A]{GrunewaldSegal80} and seems to imply that they are fixed).  For instance, Grunewald and Segal use that the group and the action are part of the input in~\cite[Corollaries~3 and 4]{GrunewaldSegal80}.  In our applications to shifts of finite type and graph $C^*$-algebras,  the particular $J$, $\rho$ and $\Gamma$ used will be dependent on the input to our decidability questions.
\end{remark}

Observe that if $\mathcal P$ is a finite poset as above, then
$GL_{\mathcal P,\vec n}(\C )$ and $SL_{\mathcal P,\vec n}(\C )$ are $\mathbb
Q$-groups.  Indeed, $GL_{\mathcal P,\vec n}(\C )$ is the subgroup of
$GL_{|\vec n|}(\C )$ defined by the polynomials over $\mathbb Z$ saying
that an entry belonging to $M_{i,j}$ with $i\npreceq j$ is $0$.  The
subgroup $SL_{\mathcal P,\vec n}(\C)$ is defined by the additional
equations stating that each diagonal block has determinant
$1$.

We let the $\Q$-group
$J=GL_{\mathcal P,\vec n}(\C)\times GL_{\mathcal P,\vec n}(\C)$
act on the $\C$ vector space
 $M_{\mathcal P,\vec n}(\C)$ by $(U,V): A \mapsto UAV^{-1}$;
this
action
is a rational action of   $J$.
 This action restricts to
 an action of $SL_{\mathcal P,\vec n}(\C )\times  SL_{\mathcal P,\vec n}(\C )$.
 Given $A$ in  $M_{\mathcal P,\vec n}(\Q)$,
 the polynomials with rational coefficients
 which compute the entries of $UAV^{-1}$ for $(U,V)$ in
$J$ can be effectively computed
%
 from $\mathcal P$ and $\vec n$. We immediately obtain the following
 corollary of Theorem~\ref{grunewaldsegal}.

%
\begin{corollary}\label{decidecor}
  Given a finite poset $\mathcal P$, a vector $\vec n$ of positive integers and matrices $A,B\in M_{\mathcal P,\vec n}(\mathbb Q)$, one can decide whether
  $A,B$ are $GL_{\mathcal P, \vec n}(\Z) $-equivalent and whether they are
  $SL_{\mathcal P, \vec n}(\Z)$-equivalent.
\end{corollary}

As noted in the introduction,
Corollary \ref{decidecor} combines with the works
\cite{mb:fesftpf,Restorff06}  to give the following.

\begin{corollary}\label{fedecidable}
Flow equivalence is decidable for shifts of finite type.
\end{corollary}

\begin{corollary} \label{stableisockdecidable}
Stable isomorphism is decidable for Cuntz-Krieger algebras satisfying
Cuntz's condition II.
\end{corollary}

\section{Rectangular matrices}

We now consider poset blocked matrices with a
rectangular structure. This is natural, and necessary
for showing that the work of \cite{errs:complete}
implies  general  decidability results  for unital
graph $C^*$-algebras.
The adjacency matrices for these (directed) graphs have only finitely many
vertices, but may have 
countably many edges; the analysis of their
adjacency matrices (with \lq\lq $\infty$\rq\rq\ an allowed entry) is reduced
in \cite{errs:complete}
to the analysis of associated rectangular matrices with integer entries.

We will use (and slightly augment) notations from \cite{errs:complete}.
Take $\mathcal P$ and $N$ as above.
Let  $\vec m$ and $\vec n$ be nonnegative elements of $\Z^N$.  Set
$\mathcal I =\{ i\colon m_i > 0\}$,
$\mathcal J =\{ j\colon n_j > 0\}$, $m=\sum m_i$, $n=\sum n_j$.
We impose the nontriviality requirement that  $\mathcal I $ and $\mathcal J$
are nonempty.
For $R$ a subring  of $\C$,
define $\mpmnr$ to be the set of $m\times n$ matrices
with $\mathcal I \times \mathcal J$ block form, with
$M_{ij}$ an $m_i \times n_j$ matrix over $R$
such that $M_{i,j}\neq 0$ implies $i\preceq j$.
(As in \cite{errs:complete}, $m_i=0$ can be viewed as
producing an empty block row indexed by $i$, and similarly
$n_j=0$ corresponds to an empty block column.)
$\mathcal I$ and $\mathcal J$ are posets, with the order inherited
from $\mathcal P$.

Given a tuple
 $\vec n$ 
over $\Z_+$, with $\mathcal J$ indexing the
indices $j$ at which $n_j>0$, and with $\mathcal J$ nonempty,
we let $\mpnr$ denote $\mpnnr$.
If all entries of
 $\vec n$
 are positive, then
this agrees with $\mpnr$ as defined earlier;
in general, $\mpnr$ is the set of matrices over $R$ with
$\mathcal J \times \mathcal J$ block structure corresponding to the
poset $\mathcal J$ and the associated positive entries of
 $\vec n$.
Let $\glpnr$ be the group of units of
$\mpnr$, with $\slpnr$ its  subgroup
of matrices $M$ such that for $n_j >0$,
we have $\det M_{jj}=1$.

\begin{example}
Let $\mathcal P =\{1,2,3,4,5\}$ be the
poset such that
$i\preceq j$ iff $i=j$, $i=1$ or  $(i,j) \in
\{
(2,5),(3,4)\}$.
Let
$\vec m = (1,0,1,0,1)$ and $ \vec n = (1,1,0,1,1)$. Then
$\mathcal I = \{1,3,5\}$ and $ \mathcal J = \{ 1,2,4,5\} $.
We display some general matrix forms:
\begin{alignat*}{3}
  & \ \  \slpmz  \quad\quad
  &&
  \quad \quad
  \mpmnz \quad\quad \quad
  && \quad \quad \glpnz  \\
& \begin{pmatrix}
  1 & * & *  \\
  0 & 1 & 0  \\
  0 & 0 & 1
\end{pmatrix}
&&
\quad
\begin{pmatrix}
  * & * & * & * \\
  0 & 0 & * & 0 \\
  0 & 0 & 0 & *
\end{pmatrix}
&&
\begin{pmatrix}
\pm  1 & * & * & * \\
0 & \pm 1 & 0 & * \\
0&  0 & \pm 1 & 0  \\
0&  0 & 0 & \pm 1
\end{pmatrix}
\end{alignat*}
in which each $*$ is an arbitrary entry from $\Z$. $\qed$
\end{example}

If $U\in \mpmr$, $M\in \mpmnr$ and
$V\in \mpnr$, then $UMV\in \mpmnr$.
The rule
$\rho_{(U,V)}\colon M \mapsto UMV^{-1}$
defines an action of the
 $\Q$-group
$\glpmc \times \glpnc$
on
$\mpmnc$.
This is an explicitly given rational action.
The next result follows immediately from
Theorem \ref{grunewaldsegal}.
\begin{corollary}\label{Hactiontheorem}
  Suppose $H$ is an explicitly given  $\Q$-group and $H$ is a subgroup of
  $\glpmc \times \glpnc$ (given by an explicit embedding defined over $\Q$).
  Then given matrices $A,B$ in
  $\mpmnq$,
    there is an algorithm which decides
 whether there exists $(U,V)$ in $H$ such that
  $UAV^{-1}=B$ (and produces such a $(U,V)$, when one exists).
  \end{corollary}

For clarity, we next address a minor point directly.

\begin{corollary}\label{opremark}
    Suppose $H$ is an explicitly given  $\Q$-group and $H$ is a subgroup of
  $\glpmc \times \glpnc$ (given by an explicit embedding defined over $\Q$).
  Then given matrices $A,B$ in
  $\mpmnq$,
  there is an algorithm which decides
 whether there exists $(U,V)$ in $H$ such that
  $UAV=B$ (and produces such a $(U,V)$, when one exists).
\end{corollary}
\begin{proof}
Let
  $H^{\circ}$ be the image of $H$ under the map $(U,V)\mapsto (U,V^{-1})$.
  The following are equivalent:
  (i) there exists $(U,V)$ in $H_{\Z}$
   with $UAV=B$;
  (ii)  there exists $(U,V)$ in $(H^{\circ})_{\Z}$
with $UAV^{-1}=B$.
  Even if $H\neq H^{\circ}$, the group $H^{\circ}$ is an explicitly
  given $\Q$-group in   $\glpmc \times \glpnc$.
  Corollary~\ref{Hactiontheorem} applies with $H^{\circ}$ in
  place of $H$, and this decides (ii).
\end{proof}

%
%
%


Eilers, Restorff, Ruiz and S{\o}rensen reduced the
problem of deciding
stable isomorphism of
two unital graph
$C^*$-algebras to
the problem of deciding, given $A,B$ in
$\mpmnz$, whether there exists $(U,V)$ in
  $ \glpmz \times \glpnz$, such that 
$V\{i\}$ (the $i$th
  diagonal block of $V$) equals $1$ whenever
  $n_i=1$, and 
$UAV=B$. (See \cite[Corollary 14.3]{errs:complete}) for this reduction.) 
 Because 
\[
    \{ (U,V) \in \glpmq \times \glpnq\colon
    n_i=1 \implies V\{ i\} =1
    \}  
\]
is a 
$\Q$-group in which the 
allowed $(U,V)$ form an explicitly given arithmetic
    subgroup, 
  their work
  implies the following result (which 
    \cite[Corollary 14.3]{errs:complete} states in terms of Morita equivalence). 

\begin{theorem} Stable isomorphism  of unital graph $C^*$-algebras is
  decidable.
\end{theorem}

Below, for $R$ a subring of $\C$
and $M$ an $n\times m$ matrix over $\C$,
$\mathrm{im}_{R} (M)$ denotes
$\{Mz\in \C^n \colon z \in R^m \}$  .

\begin{theorem}\label{decidingwithunit}
  Suppose that $x,y$ are column vectors in $\Z^n$ and $H$ is
 an explicitly given  $\Q$-group
 which is a subgroup  of   $\glpmc \times \glpnc$ (via an explicitly given
 embedding defined over $\Q$).
  Then given matrices $A,B$ in
  $\mpmnq$,
    there is an algorithm which decides
    whether there exists $(U,V)\in H_{\Z}$ such that the following hold:
  \begin{align} \label{decideunit1}
    UAV^{-1} \ &=\ B\  , \quad \text{ and } \\
\label{decideunit2}
    (V^{-1})^Tx -y \ &  \in \ \mathrm{im}_{\Z} (B^T) \  .
  \end{align}
The algorithm produces such a $(U,V)$, when one exists.
  \end{theorem}
\begin{proof}
  Corollary  \ref{Hactiontheorem} decides whether there exists
  $(U,V) \in H_{\Z}$ such that \eqref{decideunit1} holds, and if
  so produces such a $(U,V)$.
  If $(U,V)$ doesn't exist, the problem is decided; given such a $(U,V)$,
  after replacing
  $(A,B,x,y)$ with $(UAV^{-1},B, (V^{-1})^Tx, y)$, it remains
  to produce a deciding algorithm in the case $A=B$.
We leave this step to Lemma \ref{cokernellemma} below.
\end{proof}

\begin{remark}\label{re:defineoverQ}
If $C$ is an $m\times n$ integer matrix, then the set of $m\times m$ matrices $D$ with $D\cdot \mathrm{im}_{\C}(C)\subseteq \mathrm{im}_{\C}(C)$ is the vanishing set of an explicitly given set of polynomials over $\Q$ (assuming $C$ is given explicitly).  Indeed, using standard linear algebra over $\Q$, we can find a $k\times m$ matrix $M$ over $\Q$ so that $Mx=0$ if and only if $x\in \mathrm{im}_{\C}(C)$.  Then we are looking for the matrices $D$ such that $MDC=0$, which is an explicitly given set of polynomial equations over $\Q$ in the entries of $D$.
\end{remark}

\begin{lemma}\label{cokernellemma}
  Suppose $m,n$ are integers; $A$ is an $m\times n$ matrix with
  integer entries; and $J$ is an explicitly given $\Q$-group in
  $\glmc \times \glnc$. Then there is an algorithm which decides,
  given
  $x,y$ in $\Z^n$, whether there exists $(U,V)$ in $J_{\Z} $ such that
  \begin{align} \label{coklemma1}
    UAV^{-1} \ &=\ A\  , \quad \text{ and } \\
    \label{coklemma2}
    (V^{-1})^Tx -y \ &  \in \ \textnormal{im}_{\Z} (A^T) \  .
  \end{align}
  \end{lemma}
\begin{proof} Let
  $J_A = \{ (U,V)\in J\colon UAV^{-1} =A \}$, an explicitly given
  $\Q$-group. Set $H_A= \{ V\colon (U,V) \in J_A\}$ .
  Let $E=\text{End}(\mathrm{im}_{\C}(A^T))$
 be the set of
  $n\times n$ matrices $M$ over $\C$ such that
  $\{ Mw: w\in \mathrm{im}_{\C}(A^T)\} \subseteq \mathrm{im}_{\C}(A^T)$.
  Define $K_A$ to be the set of matrices $K$ in
  $GL_{n(m+1)}(\C )$ with block form
  $(K_{ij})_{0\leq i,j \leq m} $,
    with each $K_{ij}$ $n\times n$;
$K_{00}^T \in H_A$ (or equivalently, $(K_{00}^T)^{-1} \in H_A$);
    $K_{ii}= I$ for $1\leq i \leq m$;
    $K_{ij} = 0$ if $1\leq i\neq j$;
    and $K_{0j} \in E$ if $j\geq 1$.
    Visually, we have
  \begin{equation}\label{eq:matrix}
  K= \begin{pmatrix}
    K_{00}& K_{01} &K_{02} & \cdots & K_{0m} \\
    0    & I & 0 & \cdots & 0 \\
    0    & 0 & I &\cdots & 0 \\
    \vdots   &\vdots &\vdots &\vdots & \vdots \\
    0    & 0 & 0 &\cdots & I
  \end{pmatrix} \ .
  \end{equation}

  We claim that $K_A$ is a group. To show this, it suffices to show,
  given $V$ in $H_A$, that
  $V^T \in E$ and  $(V^{-1})^T \in E$.
  (This follows from considering, given
  $K,L$ in $K_A$ with   block forms from \eqref{eq:matrix},
  the block forms of $K^{-1}$ and $KL$;
  for $1\leq j \leq m$,
  $(K^{-1})_{0j} = -K_{00}^{-1}K_{0j}$ and
  $(KL)_{0j} = K_{00}L_{0j} + K_{0j}$.)
  Because $H_A$ is a group, it suffices to show
  $V^T \in E$.
For this,
  pick $U$ such that $(U,V) \in J_A$, and note
\begin{equation}\label{eq:pres.trans}
  UAV^{-1} = A \implies
  UA = AV \implies
A^T  U^T = V^T A^T \ .
\end{equation}
  Let $A^Tw$, $w\in \C^m$, be an arbitrary element of
  $\mathrm{im}_{\C}(A^T)$.
Then
  \[
  V^T \big(A^Tw\big)
  = A^T \big(U^T w \big) \in \mathrm{im}_{\C}(A^T)\ .
  \]

%
  The group $K_A$ is an explicitly given $\Q$-group by Remark~\ref{re:defineoverQ}.
  Set $W = (\C^n)^{m+1}$, writing $w$ in $W$ as
  $w= (w^{(0)}, \dots , w^{(m)})$ .
  There is an explicitly given rational action $\kappa$ of
  $K_A$ on $W$, given for $K$ in $K_A$ by the rule
  \begin{align*}
    (\kappa_K w)^{(0)} &= \sum_{j=0}^m K_{0j}w^{(j)}\ , \\
        (\kappa_K w)^{(j)} &=  w^{(j)}\ , \quad  1\leq j\leq m\ .
  \end{align*}
  Define the $\Q$-group $\widetilde K_A =
  \{ (M,K) \in \glmc \times K_A\colon (M,(K_{00}^{-1})^T) \in J_A \}$; it is explicitly given.
  There is an explicitly given rational action  $\rho$ of
  $\widetilde K_A$ on $W$ by $\rho_{(M,K)}w = \kappa_Kw$.

  Let $\Gamma$ be the subgroup of $(\widetilde K_A)_\Z$ consisting of those $(M,K)$ such that $K_{0j}\cdot \mathrm{im}_{\Z}(A^T)\subseteq \mathrm{im}_{\Z}(A^T)$ for $j\geq 1$.  We claim that $\Gamma$ is an explicitly given arithmetic subgroup of $\widetilde K_A$.  To see that it is a subgroup, we again use \eqref{eq:pres.trans}, but this time in the case that $(U,V)\in (J_A)_{\Z}$. Let $b_1,\ldots, b_k$ be a basis for the free abelian group $\mathfrak A=\mathrm{im}_{\Q}(A^T)\cap \mathbb Z^n\supseteq \mathrm{im}_{\Z}(A^T)$.  Let $b_i=A^Tc_i$ with $c_i\in \mathbb Q^m$.  Let $\ell$ be a common denominator for the entries of the $c_i$.  Then $\ell\cdot \mathfrak A\subseteq \mathrm{im}_{\Z}(A^T)$. It now follows easily that if $(M,K)\in (\widetilde K_A)_{\Z}$, then $\ell K_{0j}\cdot \mathrm{im}_{\Z}(A^T)\subseteq \ell\cdot \mathfrak A\subseteq \mathrm{im}_{\Z}(A^T)$ for $j\geq 1$.  We conclude  that $\Gamma$ has finite index in $(\widetilde K_A)_{\Z}$ and it is clearly explicitly given.

  Let $a_1^T, \dots , a_m^T$ denote the columns of $A^T$.
  Given $x,y$ in $\Z^n$, we claim that the following are
  equivalent.
  \begin{enumerate}
  \item $\exists (U,V)$ in $J_{\Z}$ such that
        $(V^{-1})^Tx -y \in \mathrm{im}_{\Z} (A^T) $  and
    $UAV^{-1} = A$.
  \item
    $\exists (M,K)$ in $\Gamma$ such that
    $\rho_{(M,K)}\colon (x,a_1^T, \dots , a_m^T) \mapsto
    (y,a_1^T, \dots , a_m^T)$.
  \end{enumerate}
  Let us check the claim. Given $(U,V)$
in $J_{\Z}$
  from (1), we have $V\in H_A$, and there are
  integers $r_1, \dots , r_m$ such that
  $(V^{-1})^Tx -y = r_1a_1^T + \dots + r_ma_m^T $ .
  Define $(M,K)$ in $\Gamma$ by setting $K_{00} = (V^{-1})^T$,
  $K_{0j} = -r_jI$ for $1\leq j \leq m$, and $M=U$.
  Then
  $\kappa_K\colon (x,a_1^T, \dots , a_m^T) \mapsto
  (y,a_1^T, \dots , a_m^T)$.

  Conversely, suppose
  $(M,K)$ in $\Gamma$, satisfies (2).
  Set $U=M$ and $V=(K_{00}^{-1})^T$.
    Then $(U,V)\in J_{\Z}$, and
    \[
(V^{-1})^T x - y =
    -K_{01}a_1^T - \cdots -K_{0m}a_m^T \in
    \textnormal{im}_{\Z}(A^T)\ .
    \]
    Because $(M,K) \in \Gamma$, we have
    $(U,V) \in J_{\Z}$. 
    This finishes the proof of the  claim.

  By Theorem \ref{grunewaldsegal}, there is an algorithm
  deciding whether
  (2) holds, because $\kappa$ is an explicitly given rational
  action of $K_A$ on $W$. Therefore there is an algorithm deciding
  (1).
\end{proof}

Eilers, Restorff, Ruiz and S{\o}rensen reduced the
problem of deciding
isomorphism of
two unital graph
$C^*$-algebras to   
the  result of Theorem \ref{decidingwithunit}
(after $V^{-1}$ in
  \eqref{decideunit1} and \eqref{decideunit2} 
is replaced with $V$).
In this application, the group $H_{\Z}$ of
  Theorem \ref{decidingwithunit}   is (again) 
\[
H_{\Z}= \{ (U,V) \in \glpmz \times \glpnz:
n_i=1 \implies V\{ i\} =1\} \ .
\] 
  The reduction to Theorem \ref{decidingwithunit} 
is explained in  the proof of \cite[Corollary 14.7]{errs:complete}. 
It follows
that their work
implies the following result.

\begin{theorem} \label{unitalcdecide}
(\cite[Corollary 14.7]{errs:complete}) 
  Isomorphism  of unital graph $C^*$-algebras is
  decidable.
\end{theorem}

\section{Isomorphism of diagrams and quiver representations}
The purpose of this section is to show that the isomorphism problem is decidable for finite diagrams of
homomorphisms of finitely generated abelian groups.
These include diagrams of the sort that appear
in full and reduced K-webs as invariants
of operator algebras or flow equivalence.

Let $Q=(Q_0,Q_1)$ be a finite directed graph, hereafter called a \emph{quiver}, with vertex set $Q_0$ and edge set $Q_1$.  We shall write $\so(e)$ and $\ra(e)$ for the source and target of an edge $e$, respectively.  By a \emph{$\mathbb Z$-representation} $\Phi=(A,\phi)$ of $Q$ we mean an assignment of a finitely generated abelian group $A_v$ to each vertex $v\in Q_0$ and a homomorphism $\phi_e\colon A_{\so (e)}\to A_{\ra (e)}$. A morphism $f\colon (A,\phi)\to (B,\rho)$ is a collection of homomorphisms $f_v\colon A_v\to B_v$, one for each $v\in Q_0$, such that the diagram
\[\xymatrix{A_{\so(e)}\ar[r]^{\phi_e}\ar[d]_{f_{\so(e)}}  &A_{\ra(e)}\ar[d]^{f_{\ra (e)}}\\ B_{\so(e)}\ar[r]_{\phi_e} &B_{\ra (e)}}\]
commutes for all $e\in Q_1$.  The category of $\mathbb Z$-representations of $Q$ will be denoted $\mathrm{rep}_{\mathbb Z}(Q)$.

The \emph{path ring} $\mathbb ZQ$ is the ring defined as follows.  As an abelian group, it has basis the set of directed paths in $Q$, including an empty path $P_v$ for each vertex $v\in Q_0$.  The product of two basis elements $q$ and $r$, is their concatenation, if defined, and otherwise is $0$. We follow the convention here of concatenating edges from right to left, as if we were composing functions.  For example, the path \[u\xrightarrow{\,e\,}v\xrightarrow{\,f\,}w\] is denoted $fe$. Note that $\mathbb ZQ$ is finitely generated as a ring by the $P_v$ with $v\in Q_0$ and the edges $e\in Q_1$.  Also note that $\mathbb ZQ$ is unital with $1=\sum_{v\in Q_0} P_v$ a decomposition into orthogonal idempotents. Let $\mathbb ZQ\text{-}\mathrm{mod}$ denote the category of (unital) left $\mathbb ZQ$-modules which are finitely generated over $\mathbb Z$. Then, analogously to the well studied case of representations of quivers over fields~\cite{assem}, there is an equivalence of categories between $\mathrm{rep}_{\mathbb Z}(Q)$ and $\mathbb ZQ\text{-}\mathrm{mod}$.  We state here how the equivalence behaves on objects because we want to show that it can be done algorithmically.  The fact that this is an equivalence of categories follows from a more general result of Mitchell~\cite[Theorem~7.1]{ringoids} applied to the free category generated by $Q$.

If $(A,\phi)$ is a $\mathbb Z$-representation of $Q$, then we obtain a left $\mathbb ZQ$-module, finitely generated over $\mathbb Z$, by taking as the underlying abelian group $\mathfrak A=\bigoplus_{v\in Q_0} A_v$.  The empty path $P_v$ acts as the projection to the summand $A_v$ (so it is the identity on $A_v$ and annihilates $A_w$ with $w\neq v$).  A non-empty path $p=e_n\cdots e_2e_1$ from $v$ to $w$ acts on $A_v$ by the composition $\phi_{e_n}\cdots \phi_{e_2}\phi_{e_1}$ and is zero on all summands $A_u$ with $u\neq v$.  Conversely, if $\mathfrak A$ is a left $\mathbb ZQ$-module, then we define $A_v=P_v\mathfrak A$ for $v\in Q_0$.  From the orthogonal decomposition $1=\sum_{v\in Q_0} P_v$ it follows easily that $\mathfrak A=\bigoplus_{v\in Q_0} A_v$ and that if $e\in Q_1$, then $eA_v=0$ if $v\neq \so(e)$ and $eA_{\so(e)}\subseteq A_{\ra(e)}$.  Thus we can define $\phi_e\colon A_{\so(e)}\to A_{\ra (e)}$ by $\phi_e(a)=ea$.

For algorithmic problems, we assume that $\mathbb Z$-representations $(A,\phi)$ of $Q$ are given by providing a finite
presentation for each group $A_v$ and giving the image under $\phi_e$ of each generator of $A_{\so(e)}$.  We assume that $\mathbb ZQ$-modules, finitely generated over $\mathbb ZQ$, are given via finite presentations as abelian groups and with the action of each edge and each empty path $P_v$ on the generators specified.  Clearly, there is a Turing machine which can turn such a presentation of a $\mathbb Z$-representation into such a presentation of a $\mathbb ZQ$-module, finitely generated over $\mathbb Z$ (and vice versa).  Therefore, to algorithmically decide isomorphism of $\mathbb Z$-representations of $Q$ is equivalent to deciding the isomorphism problem for $\mathbb ZQ$-modules which are finitely generated over $\mathbb Z$.  But Grunewald and Segal~\cite[Corollary~4]{GrunewaldSegal80} solved the isomorphism problem for $R$-modules finitely generated over $\mathbb Z$ when $R$ is a finitely generated ring. Consequently, we have the following.

\begin{theorem}\label{t:iso.quiver}
There is an algorithm that given as input a finite quiver $Q$ and two $\mathbb Z$-representations, decides whether the representations are isomorphic.
\end{theorem}

The reader is referred to~\cite{BoyleHuang} for the definitions of full and reduced $K$-webs in the following corollary.

\begin{corollary}\label{kwebglpisomorphism}
  Let $\mathcal P$ be a finite poset. There is an  algorithm which,
given     matrices in $\mpnz$,
 decides  whether
 their full $K$-webs  are isomorphic, and
 there is an algorithm which decides
 whether their reduced $K$-webs are isomorphic.
\end{corollary}

In~\cite{BoyleHuang}, for matrices $A,B$ within a
subclass of $\mpnz$
sufficient to
address problems of stable isomorphism, it was shown that
two matrices  are $\glpnz$ equivalent if and only if
the reduced $K$-webs of $I-A$ and $I-B$ are isomorphic.
Thus
Corollary \ref{kwebglpisomorphism} gives an alternate route
to proving  Corollary \ref{stableisockdecidable}.

In \cite{BoyleHuang}, there is also a characterization of flow equivalence of
shifts of finite type in terms of more refined isomorphism
relations of reduced $K$-webs. We believe that the work
of Grunewald and Segal can also be applied to show decidability
of  isomorphisms of quiver representations satisfying such constraints.
We will not attempt this here.

\appendix
\section{The flow equivalence reduction} 

The paper \cite{mb:fesftpf}, with appeal to
\cite{BoyleHuang},
 reduced the problem of deciding flow equivalence of 
shifts of finite type to the problem of deciding whether two matrices
in $ M_{\mathcal P,\vec n}(\Z)$ are $SL_{\mathcal P,\vec n}(\Z)$-equivalent.
The relevant statements in \cite{mb:fesftpf} are given in terms
of infinite matrices; we will provide details for the  translation
to the finite matrix claim. 

\begin{definition} 
  Given $\vec m \leq \vec r$, we have a natural embedding
  $\iota_{ \vec r} = \iota: M_{\mathcal P,\vec m}(\Z) \to
  M_{\mathcal P,\vec r}(\Z) $
  as follows. For $\{i,j\}\subset \mathcal P$, the map $\iota$
  embeds the $ij$ block of $M$ as the upper left corner of the
  $ij$ block of $\iota M$. Outside the embedded upper left corner,
  the $ij$ block of $\iota M$ is zero if $i\neq j$ and agrees with the
  identity matrix if $i=j$.
\end{definition}

According to \cite[Section 3]{mb:fesftpf}, flow equivalence
of shifts of finite type is decidable if there is a procedure
to answer the following question. 
\begin{enumerate} 
\item \label{infinitecriterion}
  Suppose $B\in M_{\mathcal P,\vec m}(\Z)$ and $ B'\in M_{\mathcal P,\vec{m}'}(\Z) $,  
with  $m_i=1$ iff $ m'_i =1$.
Does  there exist $\vec r$, 
 with  $r_i =1$ iff $n_i=1$,
such that 
$\iota_{\vec r} B$ and $\iota_{\vec r}  B'$ are
$\slprz$ equivalent?
\end{enumerate} 
The matrices $B,B'$ of \eqref{infinitecriterion}
  correspond to the matrices $I-A$, $I-A'$
  in condition (2) of \cite[Theorem 3.4(2)]{mb:fesftpf}.
  Our statement with $\vec r$
  is a translation of the infinite matrix statement of
  that conditon (2). 

Define $\vec n = (n_1, \dots , n_N)$ by $ n_i = 1$ if $ m_i =1$, and
otherwise $ n_i = 2 + \max \{m_i,m'_i\}$.
We claim that \eqref{infinitecriterion} holds 
if and only if it holds  for $\vec r = \vec n$.

To prove the nontrivial implication in the claim, suppose  $\vec r$ satisfies 
\eqref{infinitecriterion}.
Without loss of generality, 
we may assume  $\vec r\geq \vec n$. 
If $\vec n_i < \vec r_i$,
then  the  entries of 
the $i$th diagonal block of $\iota_{\vec n}B$
(and likewise the entries of
the $i$th diagonal block of $\iota_{\vec n}B'$) 
have greatest common divisor equal to 1. 
Then the stabilization result \cite[Corollary 4.11]{BoyleHuang}
shows the $\slprz$ equivalence of $\iota_{\vec r} B$ and $\iota_{\vec r} B'$ 
guarantees the  $\slpnz$ equivalence of $\iota_{\vec n} B$ and $\iota_{\vec n} B'$.
The definition of $\vec n$ then guarantees the Factorization
Theorem \cite[Theorem 4.4]{mb:fesftpf}
applies to produce the positive equivalence
in condition (1) of \cite[Theorem 3.4]{mb:fesftpf}.

\bibliographystyle{alpha}
\bibliography{standard5}

\end{document}